\documentclass[12pt,a4paper,reqno]{amsart}

\usepackage{amsmath,amsthm,
amssymb,tabularx}
\usepackage{mathabx}
\usepackage{comment}
\usepackage{color}
\usepackage[utf8]{inputenc}

\numberwithin{equation}{section}

\theoremstyle{plain}
\newtheorem{thm}{Theorem}[section]

\newtheorem{lm}[thm]{Lemma}

\newtheorem{prop}[thm]{Proposition}

\theoremstyle{definition}

\allowdisplaybreaks

\def\R{\mathbb{R}}

\def\Z{\mathbb{Z}}

\def\e{\varepsilon}
\let\a\alpha

\let\cal\mathcal

\def\sh{{\rm sh}}
\def\ch{{\rm ch}}
\def\arcch{{\rm arcch}}
\let\poll\l
\let\l\lambda 
\let\check\widecheck

\def \le {\leqslant}
\def \ge {\geqslant}
\let \<\langle
\let \>\rangle
\let\phi\varphi
\let\kappa\varkappa

\def\bal{\begin{align*}}
\def\eal{\end{align*}}
\def\beq{\begin{equation}}
\def\eeq{\end{equation}}

\title[Norms of certain functions of a distinguished Laplacian]{Norms of certain functions of a distinguished Laplacian on the $ax+b$ groups}

\author[R.~Akylzhanov]{Rauan Akylzhanov}
\address{Rauan Akylzhanov, School of Mathematical Sciences, Queen Mary University of London, Mile End Road, London E1 4NS, United Kingdom}
\email{akylzhanov.r@gmail.com}

\author[Yu.~Kuznetsova]{Yulia Kuznetsova}
\address{Yulia Kuznetsova, University Bourgogne Franche-Comt\'e, 16 route de Gray, 25030 Besan\c con, France}
\email{yulia.kuznetsova@univ-fcomte.fr}

\author[M.~Ruzhansky]{Michael Ruzhansky}
\address{Michael Ruzhansky, Department of Mathematics: Analysis, Logic and Discrete Mathematics, Ghent University, Krijgslaan 281, Building S8, B 9000 Ghent, Belgium, and School of Mathematical Sciences, Queen Mary University of London, Mile End Road, London E1 4NS, United Kingdom}
\email{Michael.Ruzhansky@ugent.be}

\author[H.~Zhang]{Haonan Zhang}
\address{Haonan Zhang, Institute of Science and Technology Austria (IST Austria),
	Am Campus 1, 3400 Klosterneuburg, Austria}
\email{haonan.zhang@ist.ac.at}

\subjclass[2010]{22E30; 43A15; 42B15; 35L05; 43A90} 

\begin{document}

\maketitle
\begin{abstract}
	The aim of this paper is to find new estimates for the norms of functions of the (minus) distinguished Laplace operator $\cal L$ on the `$ax+b$' groups. 
	
	The central part is devoted to spectrally localized wave propagators, that is, functions of the type 
	$\psi(\sqrt{\cal L})\exp(it \sqrt{\cal L})$, with $\psi\in C_0(\R)$. We show that for $t\to+\infty$, the convolution kernel $k_t$ of this operator satisfies
	$$
	\|k_t\|_1\asymp t, \qquad \|k_t\|_\infty\asymp 1,
	$$
	so that the upper estimates of D. M\"uller and C. Thiele (Studia Math., 2007) are sharp.
	
	As a necessary component, we recall the Plancherel density of $\cal L$ and spend certain time presenting and comparing different approaches to its calculation. Using its explicit form, we estimate uniform norms of several functions of the shifted Laplace-Beltrami operator $\tilde\Delta$, closely related to $\cal L$. The functions include in particular $\exp(-t\tilde\Delta^\gamma)$, $t>0,\gamma>0$, and $(\tilde\Delta-z)^s$, with complex $z,s$.
	
\end{abstract}


\section{Introduction}

Let $G$ denote the `$ax+b$' group of dimension $n+1\ge2$, parameterized as $G = \{ (x,y)\in\R\times\R^n\}$, with multiplication given by
$$
(x,y)\cdot (x',y') = (x+x', y+ e^x y');
$$
in the case $n=1$, it is the group of affine transformations of the real line.
In this form, the right Haar measure is just $dm_r(x,y)=dx dy$.
This group is well-known to be non-unimodular, solvable, and of exponential growth; the modular function is given by $\delta(x,y) = e^{-nx}$.

Let $X= \dfrac{\partial }{\partial x}$ and $Y_k=e^x \dfrac{\partial }{\partial y_k} $, $1\le k\le n$, denote left-invariant vector fields on $G$, given also by $\exp(tX)=(t,0)$, $\exp(tY_k)=(0,te_k)$, $t\in\R$ (where $(e_k)$ is the standard basis in $\R^n$). We consider the (minus) distinguished Laplacian $\cal L = -X^2-\sum_{k=1}^n Y_k^2$. A detailed exposition of this setting 
 can be found in \cite{hulan-cm}.

$\cal L$ is a positive self-adjoint operator on $L^2(G,m_r)$, and for any bounded Borel function $\psi$ on $\R$ one can define, by the spectral theorem, a bounded operator $\psi(\cal L)$ on $L^2(G,m_r)$. For $\psi$ good enough, this is a (right) convolution operator with a kernel $k_\psi$, for which an explicit formula is available \cite[Proposition 4.1]{mueller-thiele}. There has been much work to determine in which cases $\psi(\cal L)$ is also bounded on $L^p(G,m_r)$, $p\ne2$ (see \cite{cow-hulan} and references therein, or more recently \cite{hebisch-steger,mueller-thiele}).

Let us put this into a wider context. If we consider a connected Lie group $G$ as a Riemannian manifold, endowed with the left-invariant distance, we obtain the Laplace--Beltrami operator $\Delta$ on $G$. In the unimodular case, $\Delta$ is itself a distinguished Laplacian as defined above, that is, $\Delta = -\sum_k X_k^2$ where $(X_k)$ is a basis of the Lie algebra of $G$. If $G$ is non-unimodular, however, these two operators $\Delta$ and $\cal L$ are different and can have significantly different behaviour.

By the classical H\"ormander-Mikhlin theorem, a function of the Laplacian $m(\Delta)$ is a multiplier on $L^p(\R^n)$ if $[n/2]+1$ derivatives of $m$ decrease quickly enough. One can observe a similar behaviour on any group of polynomial growth \cite{alex}. On the other hand, on symmetric spaces of non-compact type, a class which includes the `ax+b' groups and implies exponential growth, the function $m$ must be holomorphic in a strip around the spectrum of $\Delta$ in order to generate an $L^p$-multiplier, $p\ne2$ \cite{clerc-stein, taylor, anker-ann}.

It turns out, surprisingly, that the volume growth does not determine the behaviour of the distinguished Laplacian $\cal L$ as defined above: on the $AN$ groups, in particular on `ax+b' groups, a finite number of derivatives is sufficient for an $L^p$ multiplier \cite{hebisch-pams,hebisch-steger}. One can ask then what is the class of groups carrying Laplacians of different kind, as the `ax+b' groups do. One could conjecture that these are solvable groups of exponential growth; this has been disproved however by an example of a semidirect product of $\R$ with the $(2n+1)$-dimensional Heisenberg group $\mathbb{H}^n$ \cite{christ-mueller}. This question remains open, and the criterium might be the symmetry of the group algebra $L^1(G)$ \cite{christ-mueller}.

Looking closer at the multipliers, one can ask how far goes the ressemblance between the Laplacian $\cal L$ of ``H\"ormander--Mikhlin type'' and that of $\R^n$. In particular, in applications to PDEs oscillating functions of the type
$$
A_t = \psi(\sqrt{\cal L}) \cos(t\sqrt{\cal L}), \quad B_t = \psi(\sqrt{\cal L}) \frac{\sin(t\sqrt{\cal L})}{\sqrt{\cal L}}
$$
are of special importance. With $\psi\equiv 1$, they give the solutions of the wave equation: the function $u_t = A_tf+B_tg$ solves the Cauchy problem
\begin{align*}
&\frac{\partial^2 u}{\partial t^2}  = -\cal Lu,\\
&u(0,x) = f(x), \quad \frac{\partial u}{\partial t}(0,x) = g(x),
\end{align*}
for a priori $f,g\in L^2(G)$. One cannot of course apply the H\"ormander--Mikhlin theorem in this case. By other methods, one shows that the wave propagators $\cos(t\sqrt{\Delta})$, $\dfrac{\sin(t\sqrt{\Delta})}{\sqrt{\Delta}}$ are bounded on $L^p(\R^n)$ if and only if $\big|\frac1p-\frac12\big| < \frac 1{n-1}$ \cite{peral}, so that for a general $p$ a ``localizing'' factor $\psi$ should be introduced.

With $\psi(s)=(1+s^2)^{-\alpha}$, norm estimates of $A_t,B_t$ provide estimates of $u$ in terms of Sobolev norms $\|\cdot\|_{H^{\alpha,p}} = \|(1+\Delta^2)^{\alpha/2} \cdot\|_p$ of $f$ and $g$. In this paper, we are interested in $L^1-L^1$ and $L^1-L^\infty$ estimates.
In $\R^n,n>1$ for example,
$$
\|u_t\|_\infty \le C t^{-(n-1)/2} (\|f\|_{H^{\alpha,1}} + \|g\|_{H^{\alpha,1}}),
$$
for $\alpha$ big enough. More on this topic can be found in \cite{BCD,tao}.

For the (shifted) Laplace--Beltrami operator on hyperbolic spaces, in particular for the $ax+b$ groups, it has been proved by D.~Tataru \cite{tataru} that $L^1-L^\infty$ norms of the same operators decay exponentially in time. See also \cite{ionescu,cow,pierfelice,anker-2012}.

In this context it was surprising that D.~M\"uller and C.~Thiele \cite{mueller-thiele} did not get any decay in time for the distinguished Laplace operator on the $ax+b$ group: they show that $L^1$ norms of $A_t,B_t$ are bounded by $C(1+|t|)$, for $\psi\in C_0(\R)$ quickly decreasing, while their $L^1-L^\infty$ norms are just bounded by a constant, for a compactly supported $\psi$.

In the present paper, we prove that these estimates are actually sharp at $t\to+\infty$, obtaining lower bounds for $L^1$ and $L^1-L^\infty$ norms of $A_t,B_t$.
 In particular, $L^1-L^\infty$ estimates do not demonstrate any dispersive-type effect, confirming a conjecture of \cite{mueller-thiele}.  The function $\psi$ is supposed nonzero with certain decrease rate.
For the uniform estimate, there is no need to suppose $\psi$ compactly supported, so that both estimates are applicable to $\psi(s) = (1+s^2)^{-\alpha}$. The main part of the proofs is concentrated in Section \ref{sec-lower}.

Generalizing \cite{mueller-thiele}, M\"uller and M.~Vallarino \cite{mueller-vall} obtained upper estimates for the kernels of analogous operators $A_t$, $B_t$ on general Damek--Ricci spaces. Their sharpness is yet unknown.

We thank J.-Ph.~Anker for indicating us the following result \cite{anker-2011}, valid as well in the more general situation of Damek--Ricci spaces: The convolution kernel of $\exp(it\cal L)$ is not bounded for any nonzero $t\in\R$, so that the dispersive $L^1-L^\infty$ estimates do not hold for the Schr\"odinger equation associated to the distinguished Laplacian $\cal L$ on these spaces. However, Strichartz estimates do hold, in a weighted form. This suggests that for the wave equation Strichartz estimates might hold as well.

The rest of the paper is organized as follows. In Section 2, we collect necessary notations and conventions.

To explain the sequel, let us denote by $\tilde\Delta$ the shifted Laplace--Beltrami operator (more details in Section \ref{sec-measure}). The two Laplacians are linked by $\tilde\Delta = \delta^{-1/2}\cal L\delta^{1/2}$ (here $\delta$ stands for multiplication by the modular function), and the kernels of $\psi(\tilde\Delta)$ and  $\psi(\cal L)$ are related as $\tilde k_{\psi} = \delta^{-1/2} k_{\psi}$, both acting by right convolution.

It is known that the $L^2$-norm of $\tilde k_\psi$ is the same as of $k_\psi$ and  
is given by the integral $$\|k_\psi\|_{L^2(G,m_r)}^2 = \int_\R |\psi|^2\rho$$
with a density $\rho$ expressed via the Harish--Chandra $c$-function \cite{helgason,cow-hulan}. This appears as a building block in several multiplier estimates \cite{cow-hulan,hebisch-edin,hebisch-steger}. In Section \ref{sec-measure}, we take time to write down an explicit formula for the spectral density of the Laplacian and to relate to each other several ways to obtain it.
We put forward the fact, seemingly not discussed before, that the same density can be used also to estimate uniform norms of $\tilde k_\psi$.

In Section \ref{sec-upper}, we show that this method makes it an easy calculation to obtain exact asymptotics of the uniform norms. We find it first for $\exp(-t\tilde \Delta^\gamma)$, $\gamma>0$. Further we pass to the more technical task of estimating the norms of $(\tilde \Delta-z)^{-s}$ with $z,s$ complex ($z$ outside of the spectrum of $\tilde \Delta$, and $\Re s>(n+1)/2$). J.-P.~Anker\cite{anker-duke} has shown, among others, that the kernel of $(\tilde \Delta-z)^{-s}$ (with $z$ outside of the positive half-line) is bounded if and only if $\Re s\ge(n+1)/2$, $s\ne(n+1)/2$. 
Going rather explicitly into integration with the help of special functions, we get estimates of actual uniform norms of these functions in the same region except for the border, and obtain asymptotic bounds for $\Im z\to0$ and for $|z|\to\infty$.

The last Section \ref{sec-lower} contains, as mentioned above, the lower estimates of $L^1$ and uniform norms of the operators $A_t,B_t$ above.

\section{Notations and conventions}

We have chosen to work (mainly) with the right Haar measure $dm_r=dxdy$ and the right convolution:
$$
(f*_r g)(x) = \int_G f(xy) g(y^{-1}) \delta(y) dm_r(y).
$$
One can alternatively opt to the left Haar measure $dm_l = e^{-nx} dx\, dy$, so that $dm_r=\delta^{-1} dm_l$, and the left convolution
$$
(f*_l g)(x) = \int_G f(y) g(y^{-1}x) dm_l(y).
$$
For a function $f$ on $G$, denote $\check f(g)=f(g^{-1})$. For every $p$, the map $\tau:f\mapsto \check f$ is an isometry from $L^p(G,m_r)$ to $L^p(G,m_l)$ (or back). If $R_k:f\mapsto f*_r k$ is a right convolution operator with the kernel $k$, then $\tau R_k \tau = L_{\check k}$ is a left convolution operator with the kernel $\check k$, and
\begin{equation*}
\|L_{\check k}:L^p(G,m_l)\to L^p(G,m_l)\|=\|R_{k}:L^p(G,m_r)\to L^p(G,m_r)\|.
\end{equation*}
This means that choosing one or another convention changes nothing in norm estimates.

Below we do not use the symbol $*_r$ anymore and write just $*$ instead. We will use $\|\cdot \|_p$ and $L^p(G)$ to denote $\|\cdot \|_{L^p(G,m_r)}$ and $L^p(G,m_r)$, respectively. 

\bigskip
{\bf The distance}.
The left-invariant Riemannian distance on $G$ is given by $d\big((x,y),(0,0)\big) =: R(x,y)$, where
\begin{equation}\label{R}
R(x,y) = \arcch\Big( \ch x + \frac12 |y|^2 e^{-x} \Big).
\end{equation}
This implies in particular that $|x|\le R$ and
$$
|y|^2 = 2 e^x (\ch R - \ch x) \le 2 e^{x+R}. 
$$

\bigskip
{\bf The Plancherel weight}.
Every bounded left-invariant operator on $L^2(G,m_r)$ belongs to the right von Neumann algebra $VN_R(G)$, defined as the strong operator closure of the set of right translation operators. This applies, in particular, to $\psi(\cal L)$ with a bounded function $\psi$. Similarly, every right-invariant operator on $L^2(G,m_l)$ belongs to the left von Neumann algebra $VN_L(G,m_l)$.

On $VN_L(G)$, we have the Plancherel weight $\phi$ (see more in \cite{pedersen}), which can be viewed as an integral of operators. For an operator $L_k$ of left convolution with the kernel $k\in L^2(G,m_l)\cap L^1(G,m_l)$, one has
$$
\phi(L_k^*L_k) = \|k\|_{L^2(G,m_l)}^2.
$$
If $k=g^**g$ for some $g$ and is moreover continuous, then $\phi(L_k) = k(e)$.

Almost all the literature on Plancherel weights assumes this $VN_L(G)$ convention, but one can also consider a similar weight $\phi_r$ on $VN_R(G)$ by setting
\begin{equation}\label{phi-plancherel}
\phi_r(R_kR_k^*) = \|k\|_{L^2(G,m_r)}^2
\end{equation}
for an operator of right convolution with the kernel $k$. The two algebras are isomorphic by $A\mapsto \tau A\tau$, and $\phi_r(A) = \phi(\tau A\tau)$. This isomorphism transfers also, as mentioned above, $R_k\in VN_R(G)$ to $L_{\check k}\in VN_L(G)$, where $\check k(g) = k(g^{-1})$.

\section{Spectral measure of the Laplacian}\label{sec-measure}

The aim of this section is to write explicitly the Plancherel measure of the distinguished Laplace operator and show the relation between several apparently different approaches to its calculation.

\subsection{Plancherel measure and the Harish-Chandra $c$-function}\label{sec-Harish-Chandra}

The Plancherel measure for the spherical transform on a connected semisimple Lie group is given by a so called $c$-function found by Harish-Chandra. Many faces of this function are described in an excellent survey \cite{helgason}, of which we will need only a few facts.

Recall first that the Laplace--Beltrami operator $\Delta$ has a spectral gap: its spectrum is $[\sigma,+\infty)$ where $\sigma$ is a constant depending on the group, $\sigma=\frac{n^2}4$ in our case. For this reason, closer to $\cal L$ in its properties is the shifted operator $\tilde\Delta = \Delta-\sigma$ which has $[0,+\infty)$ as its spectrum. The two operators are now linked by $\tilde\Delta = \delta^{-1/2}\cal L\delta^{1/2}$ (here $\delta$ stands for multiplication by the modular function).

Next, it is known that the $L^2$-norm of a radial function on $G$ can be expressed via the Harish-Chandra $c$-function \cite{helgason,cow-hulan}.
This can be applied to the convolution kernels of $\tilde\Delta$ and functions of it, as these kernels are radial. If $\tilde k_f$ is the kernel of $f(\tilde\Delta)$, we have
$$
\|\tilde k_f\|_2^2 = c_G \int_\R |f(\lambda^2)|^2 |c(\lambda)|^{-2}d\lambda.
$$
If $k_f$ is the right convolution kernel of $f(\cal L)$, then \cite{cow-hulan} $k_f=\delta^{1/2}\tilde k_f$ and $\|k_f\|_{L^2(G,m_r)}
=\|\tilde k_f\|_{L^2(G,m_r)}$ (this equality is verified by a direct calculation, knowing that $\tilde k_f(x^{-1})=\tilde k_f(x)$), so that the formula above is valid for $k_f$ as well.
An explicit formula of the $c$-function is known, and for the $n$-dimensional `ax+b' group it is as follows \cite{helgason}: for $\lambda\in\R$,
$$
c(\l) = \frac{c_0 2^{-i\l} \Gamma(i\l)} { \Gamma(\frac12 (\frac12 n + 1 + i\l))
 \Gamma( \frac12 (\frac12 n + i\l) )}.
$$
Since $\Gamma(z) \Gamma(z+\frac12) = \sqrt\pi 2^{1-2z} \Gamma(2z)$ \cite[1.2]{BE}, this simplifies up to
\begin{align*}
c(\l) 
= c_1\frac{\Gamma(i\l)} {\Gamma(\frac n2 + i\l)},
 \end{align*}
with $c_1=\pi^{-1/2} 2^{n/2-1}c_0$. Denote $\rho(u) = c_1^2 u^{-1/2} |c(\sqrt u)|^{-2}$, so that
$$
\|k_f\|_2^2 = C \int_\R |f(u)|^2 \rho(u) du
$$
(we will also write $\rho_n$ to indicate the dimension).
If $n=2l$ is even, $1/c$ is a polynomial: 
$$
\frac1{c(\l)} = c_1^{-1} \prod_{j=0}^{l-1}(j+i\l),
$$
so that
$$
\rho_{2l}(u) = \sqrt u \prod_{j=1}^{l-1} (j^2+u).
$$
If $n=2l+1$ is odd, we have
$$
\frac1{c(\l)} = c_1 \prod_{j=0}^{l-1}(j+\frac12+i\l) \frac{\Gamma(\frac12+i\l)} {\Gamma(i\l)},
$$
and
$$
\rho_{2l+1}(u) = u^{-1/2} \prod_{j=0}^{l-1} \Big(\big(j+\frac12\big)^2+u\Big) \,\frac{|\Gamma(\frac12+i\sqrt u)|^2} {|\Gamma(i\sqrt u)|^2},
$$
so that $\rho_{2l+1}(u) = \displaystyle\prod_{j=0}^{l-1} \big((j+1/2)^2+u\big) \rho_1(u)$.
Moreover, the reflection formula $\Gamma(z)\Gamma(1-z) = \pi/\sin(\pi z)$ and the conjugation identity $\overline{\Gamma(z)}=\Gamma(\bar z)$ imply that for real $v$,
$$
\frac{|\Gamma(\frac12+iv)|^2} {|\Gamma(iv)|^2}
= \frac{-\pi iv \sin(\pi iv)} {\pi \cos(\pi iv))} = \frac{v\, \sh(\pi v)} {\ch(\pi v)},
$$
and
$$
\rho_{2l+1}(u) = \prod_{j=0}^{l-1} \Big(\big(j+\frac12\big)^2+u\Big) \,\frac{\sh(\pi\sqrt u)} {\ch(\pi\sqrt u)}.
$$

These formulas imply that
 $\rho(u)\sim u^{(n-1)/2}$ as $u\to+\infty$, and $\rho(u)\sim u^{1/2}$ as $u\to0$. These latter estimates have been used in application to the norm estimates of the Laplacian and its functions in \cite{cow-hulan, hebisch-edin, hebisch-steger}.

\subsection{Uniform norms of the kernels}

Let us consider a right convolution operator on $L^2(G,m_r)$, $R_k:f\mapsto f*k$ (all convolutions here are taken with respect to the right Haar measure). Its adjoint is $R_k^* = R_{k^*}$, where the (right) involution is defined as $f^{*}(x) = \overline{ f(x^{-1})} \delta(x)$. The composition of a pair of operators is $R_kR_h = R_{h*k}$.

Suppose now that $0\le f=|g|^2$ and $k_f, k_g$ are the convolution kernels of $f(\cal L), g(\cal L)$, with $k_g\in L^2(G)$. We have then $f(\cal L)=g(\cal L)^*g(\cal L)=g(\cal L)g(\cal L)^*$, so that $k_f = k_g*k_g^* = k_g^* * k_g$ and
\begin{align*}
k_f(e) &= \int_G k_g(y) k_g^*(y^{-1}) \delta(y) dm_r(y) 
\\&= \int_G k_g(y) \overline{k_g(y)} \delta(y^{-1}) \delta(y) dm_r(y) = \|k_g\|_{L^2(G,m_r)}^2.
\end{align*}
By the results above,
\begin{equation}\label{k(e)}
k_f(e) = \|k_g\|_{L^2(G,m_r)}^2 = c\int_\R |g|^2\rho = c\int_\R f\rho. 
\end{equation}
This formula is thus valid for $0\le f\in L^1(\R,\rho)$. For $f$ not necessarily positive, the formula $k_f(e) = c\int_\R f\rho$ is still valid, by linearity.

It is well known that the uniform norm of a positive definite function is attained at the identity. Given a function $h\in L^2(G,m_r)$, the convolution $h^* *\, h$ is in general {\bf not} positive definite. However if we multiply it by $\delta^{-1/2}$, it becomes such, since this is a coefficient of the right regular representation of $G$ on $L^2(G,m_r)$:
\begin{align*}
(h^* * h)(x) \delta^{-1/2}(x) 
&= \delta^{-1/2}(x) \int_G \overline{ h(y^{-1}x^{-1}) } \delta(xy) h(y^{-1}) \delta(y) dm_r(y)
\\&= \delta^{1/2}(x) \int_G \overline{ h(yx^{-1}) } \delta(y^{-1}) h(y) dm_r(y)
\\&= \delta^{1/2}(x) \int_G \overline{ h(z) } \delta(x^{-1}z^{-1}) h(zx) dm_r(z)
\\&= \int_G h(zx) \delta^{-1/2}(zx) \overline{ h(z) } \delta^{-1/2}(z) dm_r(z)
\\&= \langle R_x ( h\delta^{-1/2}), h\delta^{-1/2}\rangle_r,
\end{align*}
where $\langle\cdot ,\cdot\rangle _r$ denotes the inner product induced by the right Haar measure.
This means that the kernel $\tilde k_f = \delta^{-1/2} k_f$ of $f(\tilde\Delta)$ is in this case positive definite. In particular, its uniform norm is attained at the identity:
\begin{equation}
\|\tilde k_f\|_\infty = \tilde k_f(e) = k_f(e) = c\int_\R f\rho,
\end{equation}
$0\le f\in L^1(\R,\rho)$. 
If $f$ is real-valued, but maybe not positive, then we can decompose it $f=f_+-f_-$ into positive and negative part, and get the following estimate:
\begin{equation}\label{uniform}
\|\tilde k_f\|_\infty \le \|\tilde k_{f_+}\|_\infty+\|\tilde k_{f_-}\|_\infty = \tilde k_{f_+}(e)+\tilde k_{f_-}(e) = c\int_\R 
(f_+ + f_-)\rho = c\int_\R |f|\rho.
\end{equation}
Finally, for $f$ complex-valued, we have to add a factor of $\sqrt 2$ in the right hand side:
$$
\|\tilde k_f\|_\infty \le \|\tilde k_{\Re f}\|_\infty+\|\tilde k_{\Im f}\|_\infty \le c\int_\R (|\Re f|+|\Im f|)\rho \le c\sqrt
2 \int_\R |f|\rho.
$$

Thus, the Plancherel measure helps to calculate not only $L^2$, but also uniform norms.

\subsection{Connection with the Plancherel weight and $L^2$-norms of the resolvent kernels}

In the case $n=1$ or $2$, the kernels $k_\lambda$ of the resolvent $R_\lambda=(\cal L -\lambda)^{-1}$ are 2-summable, as seen from the bounds above. This allows to obtain the Plancherel measure in a different way.

Comparing the formulas \eqref{phi-plancherel} and \eqref{k(e)}, we notice that
$$
\phi_r\big(f(\cal L)\big) = c\int_0^\infty f d\rho,
$$
for $0\le f\in L^1(\R,\rho)$. This means that the measure $\rho$ can be found from this equality, if we are able to determine $\phi_r\big(f(\cal L)\big)$.

Let $A$ be, in general, a positive self-adjoint unbounded operator on $L^2(G,m_r)$ with the spectral decomposition $A=\int_{0}^{\infty}udE(u)$. For a bounded measurable function $f$, one defines
$$
f(A) = \int_0^\infty f(u) \,dE(u)
$$
by the spectral theorem for self-adjoint unbounded operators. By construction, for all $x,y\in H=L^2(G,m_r)$
$$
\langle f(A)x,y\rangle_r = \int_0^\infty f(u) \,dE_{x,y}(u)
$$
with respect to the measure $E_{x,y}: X\mapsto \langle E(X) x,y\rangle$, $X\subset\R$ Borel.
If $A$ is left-invariant, that is, affiliated to the right group von Neumann algebra $VN_R(G)$, then $f(A)\in VN_R(G)$.

For a vector state $\zeta_{x,x}(\cdot) = \langle \cdot \,x,x\rangle$ on $VN_R(G)$, this gives already its value on $f(A)$.
Any weight $\phi_r$, and in particular the Plancherel weight is the sum of a family of normal positive functionals \cite{haag-weights}, $\phi_r = \sum_\a \phi_{r,\a}$.
Every $\phi_{r,\a}$ can be decomposed into a countable sum of positive vector states 
$\phi_{r,\a} = \sum_n \zeta_{x_{n,\a},x_{n,\a}}$, which implies that for positive $f$
$$
\phi_r\big( f(A) \big) = \sum_\a \int_0^\infty f(u) \, d\phi_{r,\a}(E(u))
= \int_0^\infty f(u) \, d\phi_r(E(u)).
$$
Linearity allows to extend this equality to arbitrary $f$, real or complex.

One can find the spectral measure of $A$ as the strong limit \cite[XII.2, Theorem 10]{dunford}
\begin{align*}
E_{[a,b]}
&= \frac1{2\pi i}\lim_{\e\to 0+} \int_a^b \big( R_{\lambda+i\e} - R_{\lambda-i\e}\big) d\lambda,
\end{align*}
where, as usual, $R_z=(A-z)^{-1}$ is the resolvent of $A$.
In the case of a positive operator, we can get nontrivial values of course only for $a\ge 0$.
We can transform
$$
R_{\lambda+i\e} - R_{\lambda-i\e} = 2i\e R_{\lambda+i\e} R_{\lambda-i\e} =
 2i\e R_{\lambda+i\e} R_{\lambda+i\e}^*.
$$
If we return now to the Laplace operator, then $R_{\lambda+i\e}=R_{k_{\lambda+i\e}}$ is a convolution operator with a kernel $k_{\lambda+i\e}$. By the lower semi-continuity of the right Plancherel weight $\phi_r$,
\begin{align*}
\phi_r(E_{[a,b]}) &\le \frac1{2\pi i}\lim_{\e\to 0+} \int_a^b 2i\e \,\phi_r( R_{k_{\lambda+i\e}} R_{k_{\lambda+i\e}}^*) d\lambda
\\ &= \frac1{\pi }\lim_{\e\to 0+} \e \int_a^b \|k_{\lambda+i\e}\|_{L^2(G,m_r)}^2 d\lambda.
\end{align*}
Using estimates in \cite{hulan} in the case $n=1$ and \cite[Lemma 3.1]{mueller-thiele} for general $n$, one can show (this is however quite a technical task) that this limit is finite if (and only if) $n\le2$, and in this case bounded by
\begin{align*}
\phi_r(E_{[a,b]}) &\le d_n \int_a^b (1+\sqrt{|\lambda|})^{n+1} d\lambda,
\end{align*}
which proves that the spectral measure is absolutely continuous with density bounded by a polynomial. We see however that this bound is not sharp.

\subsection{Application of the explicit formula for the convolution kernel}

Yet another approach is to use the explicit formula for the kernel \cite[Proposition 4.1]{mueller-thiele}: for a function $\psi\in C_0(\R)$, the kernel $k_\psi$ of $\psi(\cal L)$ is given, for any integer $l>\frac n2-1$, by
\begin{equation}\label{k-psi}
k_\psi(x,y)
 = \frac{c_l}2 e^{-\frac n2 x} \int_0^\infty \psi(u)[F_{R(x,y),l}(\sqrt u)-F_{R(x,y),l}(-\sqrt u)]\,du,
 \end{equation}
where $R(x,y)$ is given by \eqref{R},
$$
F_{R,l}(u) = \int_R^\infty D^l_{\sh, v}(e^{iuv}) (\ch v-\ch R)^{l-\frac n2} dv,
$$
$D^l_{\sh, v}$ denotes the $l$-th composition of $D_{\sh,v}$,
$D_{\sh, v}(f) = \dfrac d{dv}\Big(\dfrac f{\sh v}\Big)$ and
$$
c_l = \frac{ (-1)^l 2^{-1-\frac n2} \pi^{-\frac n2} } {i\pi \Gamma( l+1-\frac n2)}.
$$
In particular, for $x=y=0$ we have (note that $ic_l\in\R$)
\begin{align*}
k_\psi(e)
& = \frac{c_l}2 \int_0^\infty \psi(u)[F_{0,l}(\sqrt u)-F_{0,l}(-\sqrt u)]\,du
\\& = i c_l \int_0^\infty \psi(u)
\int_0^\infty D^l_{\sh, v}\big(\sin(v \sqrt u) \big) (\ch v-1)^{l-\frac n2} dv\,du.
 \end{align*}
As we have seen before, for $0\le \psi\in L^1(\R,\rho)$ this is also equal to $c\int \psi\rho$, so that (as $L^1(\R,\rho)\cap C_0(\R)$ is dense in $L^1(\R,\rho)$)
\begin{equation}\label{rho-F}
\rho(u) = \frac{c_l}{2c} [F_{0,l}(\sqrt u)-F_{0,l}(-\sqrt u)]
=\frac{ic_l}{c}\int_0^\infty D^l_{\sh, v}\big(\sin(v \sqrt u) \big) (\ch v-1)^{l-\frac n2} dv.
\end{equation}
One should note that $F_{R,l}$ is not bounded at 0, but the difference above is, and tends to 0 as $u\to 0$.

This equality is not obvious, but is easy to check in the case $n=2l=2$: we have to change the sign as $ic_l = -\frac1{4\pi^2}<0$, and then obtain
$$
-\big(F_{0,1}(\sqrt u)-F_{0,1}(-\sqrt u)\big)
 = - \int_0^\infty \frac d{dv} \Big(\frac{\sin(v \sqrt u)}{\sh v} \Big)dv
 = \sqrt u = \rho_{2}(u).
$$

It is clear that best bounds can be obtained for $l$ chosen so that $l-n/2\in\{0,-1/2\}$, and below we assume this choice.

Denote $D_{l,u}(v) = D^l_{\sh, v}(e^{iuv}-e^{-iuv})$. For $l=0$, $D_{0,u}(v) = 2i\sin(uv)$ is an odd function of $v$. By induction, one verifies that every $D_{l,u}$ is odd, too: if $f$ is odd (and analytic), then $f(v)/\sh v$ is even (and well defined at 0) and $D_{\sh,v}(f)=\frac d{dv}(f(v)/\sh v)$ is odd. This implies, in particular, that $D_{l,u}(0)=0$. Note at the same time that $D^l_{\sh, v}(e^{iuv})$ alone can be unbounded as $v\to0$. Thus, the integral below converges:
\begin{equation}\label{rho-F1}
\rho(u^2) = \frac{ic_l}{2c} \int_0^\infty D_{l,u}(v) (\ch v-1)^{l-\frac n2} dv.
\end{equation}
If $l=\frac n2$ ($n$ even), then
\begin{align*}
\rho(u^2) &= \frac{ic_l}{2c}  \int_0^\infty \frac d{dv} \Big( \frac{ D_{l-1,u}(v)}{\sh v} \Big) dv
\\& =- \frac{ic_l}{2c}   \lim_{v\to0} \frac{ D_{l-1,u}(v)}{\sh v}
=-\frac{ic_l}{2c}   \frac d{dv}D_{l-1,u}\Big|_{v=0}.
\end{align*}

Decomposing $D_{l,u}$ into its Taylor series, one can show
that $D_{l,u}'(0)=P_{0,l}(u)$ is a polynomial of degree $2l+1$.

On this way, one can obtain estimates $\rho_n(u)\le C(u^{1/2}+u^{n/2})$, however already known from the calculations of the $c$-function. 

\section{Upper norm estimates}\label{sec-upper}

\subsection{Estimates of $L^p$--$L^q$ norms}\label{sec-lp-lq}

Having at our disposal uniform and $L^2$ norms of the kernel $k$ of a convolution operator $A$ allows one to estimate its weighted $L^p-L^q$ norms in the following cases.

Suppose that $1<p<q$ are such that $1/r:=1/p-1/q\ge 1/2$.
By the Young's inequality (valid for any $p,q,r'\in[1,+\infty]$ with $1/q+1=1/r'+1/p$, which is equivalent to our condition $1/r=1/p-1/q$; recall that we are using the right Haar measure),
$$
\|(A\delta^{-1/r})f\|_{L^q(G)} = \|(\delta^{-1/r}\!f)*k\|_{L^q(G)} \le \|k\|_{r'} \|f\|_p, 
$$
$f\in L^p(G),k\in L^{r'}\!(G)$, and
if $r'\ge 2$, we can estimate
\begin{equation}\label{Lp-Lq}
\|A\delta^{-1/r}\|_{L^p\to L^q} \le \|k\|_{r'}\le \|k\|_\infty^{1-\frac{2}{r'}} \|k\|_2^{\frac{2}{r'}}.
\end{equation}
This can be viewed alternatively as the norm of $A: L^p(G,\delta^{-1/r}m_r)\to L^q(G,m_r)$.

\subsection{Fractional exponents}

This formula makes asymptotics of certains norms an easy calculation. We could recover for example known uniform bounds for the heat kernel. More generally, set $A_t= e^{-t\tilde \Delta^\gamma}$ with $\gamma>0$, and denote by $q_t$ its convolution kernel.
\def\th{{\rm th}}
For $n=2l+1$ odd,
\begin{align*}
\|q_t\|_\infty &= c\int_0^\infty e^{-tu^\gamma} \,\th(\pi\sqrt u) \prod_{j=0}^{l-1} \Big(\big(j+\frac12\big)^2+u\Big) du
\\& =
\frac c{\gamma}\, t^{-1/\gamma} \int_0^\infty e^{-x} \th\Big(\pi\Big(\frac xt\Big)^{1/2\gamma} \Big) \prod_{j=0}^{l-1} \Big(\big(j+\frac12\big)^2+\Big(\frac xt\Big)^{1/\gamma}\Big) x^{1/\gamma-1} dx.
\end{align*}
When $t\to+\infty$, $\th\Big(\pi\Big(\dfrac xt\Big)^{1/2\gamma}\Big)\sim \pi\Big(\dfrac xt\Big)^{1/2\gamma}$ (on finite intervals, which is sufficient) and $x/t\to0$, so that
$$
\|q_t\|_\infty \sim t^{-3/2\gamma} \,\frac{c\pi}\gamma \prod_{j=0}^{l-1} \big(j+\frac12\big)^2 \int_0^\infty e^{-x} x^{-1+3/2\gamma}\, dx = C t^{-3/2\gamma}.
$$
When $t\to0$, $\th\Big(\pi\Big(\dfrac xt\Big)^{1/2\gamma}\Big)\to1$ and
$$
\|q_t\|_\infty \sim t^{-1/\gamma} \frac{c\pi}\gamma \int_0^\infty e^{-x} \Big(\frac xt\Big)^{l/\gamma} dx = Ct^{-(l+1)/\gamma} = C t^{-(n+1)/2\gamma}.
$$
For $n=2l$ even, estimates are similar and lead to the same exponents of $t$. In the case $\gamma=1$, this agrees with the known bounds for the heat kernel of Davies and Mandouvalos \cite[Theorem 3.1]{davies-mand} (of course more general as they concern pointwise estimates).

Since $\|\tilde k_t\|^2_2=\|\tilde k_{2t}\|_\infty$, the formula \eqref{Lp-Lq} implies also, for $1/r=1/p-1/q\ge1/2$,
\begin{align*}
\|A_t\|_{L^p(\delta^{-1/r})\to L^q} &\lesssim C_\gamma  \begin{cases} t^{-\frac3{2r\gamma}}, & t\to+\infty,\\
  t^{-\frac{n+1}{2r\gamma }}, & t\to 0.\end{cases}
\end{align*}

\subsection{Rational functions}

Our next aim is to estimate the norm of $f(\tilde\Delta) = (\tilde\Delta-z)^{-s}$ with $z$ outside of $[0,+\infty)$. Anker \cite{anker-duke} has shown that the convolution kernels of these operators are bounded if and only if $\Re s\ge \frac{n+1}2$. Below, we estimate their actual uniform norms.
We get bounds in the range $\Re s>\frac{n+1}2$, which is the same half-plane as in \cite{anker-duke}, except for the border.

One should mention also the subsequent work of Anker and Ji \cite{anker-ji} where pointwise bounds of these kernels are obtained. The authors do not seek uniform norms or for dependence on $z$, but for real $z,s$ one could derive exact values of $\|k_{z,s}\|_\infty$ from the uniform norms of the heat kernel, using formulas in the proof of \cite[Theorem 4.2.2]{anker-ji}.

Let $k_{z,s}$ denote the convolution kernel of $f(\tilde\Delta)$. According to \eqref{uniform},
$$
\|k_{z,s}\|_\infty \le C_n \int_0^\infty \big|(u-z)^{-s}\big| \rho(u) du.
$$
We will estimate
$$
\big| (u-z)^{-s}\big| 
= |u-z|^{-\Re s} \exp(\arg(z+u) \Im s) \le |u-z|^{-\Re s} \exp(\frac\pi2 |\Im s|),
$$
which is of course not optimal, but can make us lose at most a factor of $\exp(\pi |\Im s|)$; we are not studying exact dependence on $s$, so we accept this lack of precision.

From the asymptotics of $\rho$, it is clear that the integral converges if and only if $\Re s>\frac{n+1}2$, which we assume below.

We consider the following cases:

{\bf Case 1: $\Re z\ge0$}, and we are especially interested in the asymptotics $\Im z\to0$. Let us write $z=a+ib$ with real $a,b$. The main term in the integral $\int_0^\infty |u-z|^{-\Re s} \rho(u)du$ is
\begin{align*}
\int_{\max(0,a-1)}^{a+1} |u-z|^{-\Re s} \rho(u) du &\le 2C_n (a+1)^{\frac{n-1}2} \int_0^1 |x-ib|^{-\Re s}dx
\\ &\le 2C_n (a+1)^{\frac{n-1}2} \int_0^1 (x^2+b^2)^{-\Re s/2}dx.
\end{align*}
Though it is possible to estimate this integral by elementary functions, we want to keep control on the dependence on $s$, so we are using special functions below.

Denote $\sigma = \Re s/2$. By assumption, we have $\sigma>1/2$. The integral is calculated with the help of the hypergeometric function ${}_2F_1$ \cite[2.1.3]{BE}:
\begin{align*}
I=\int_0^1 (x^2+b^2)^{-\sigma}dx &= \frac12 \int_0^1 (t+b^2)^{-\sigma}t^{-\frac12} dt
\\& = |b|^{-2\sigma}\, {}_2F_1\big(\sigma, \frac12; \frac32;-\frac1{b^2}\big).
\end{align*}

If $\sigma-1/2\notin\Z$, this equals \cite[2.1.4 (17)]{BE}
\begin{align*}
I &= |b|^{-2\sigma}\,\Big[ 
\frac{\Gamma(\frac12-\sigma)\Gamma(\frac{3}{2})}{\Gamma(\frac12)\Gamma(\frac{3}{2}-\sigma)} |b|^{2\sigma}  {}_2F_1\big(\sigma, \sigma-\frac12; \sigma+\frac12;-b^2\big) 
+ \frac{ \Gamma(\sigma-\frac12)\Gamma(\frac{3}{2})}{\Gamma(\sigma)} |b| \Big]
\end{align*}
(since ${}_2F_1\big(\frac12, 0; \frac32-\sigma;-b^2\big)\equiv1$) and at $b\to0$, is equivalent to
$$
\frac I {\Gamma(\frac{3}{2})} \sim \dfrac{ \Gamma(\sigma-\frac12)}{\Gamma(\sigma)} |b|^{1-2\sigma}
= \dfrac{ \Gamma(\frac{\Re s-1}2)}{\Gamma(\frac{\Re s}2)} |b|^{1-\Re s}.
$$

If $\sigma-1/2=m\in\Z$ (the smallest possible value is 1), we get by symmetry of the first two arguments of ${}_2F_1$ and by \cite[2.1.4 (20)]{BE}
\begin{align*}
I = |b|^{-2\sigma}\, {}_2F_1\big(\sigma, \frac12; \frac32;-\frac1{b^2}\big)
& = |b|^{-2\sigma}\, {}_2F_1\big(\frac12, \sigma; \frac32;-\frac1{b^2}\big)
\\& = |b|^{-2\sigma}\, \Big(1+\frac1{b^2}\Big)^{1-\sigma} {}_2F_1\big(1, \frac32-\sigma; \frac32;-\frac1{b^2}\big)
\end{align*}
where $3/2-\sigma=1-m$ is a nonpositive integer, so that the hypergeometric series is a polynomial. Leaving only the leading term (as $b\to0$), we obtain \cite[2.1.4 (2)]{BE}
\begin{align*}
I &\sim \frac1{b^2} \frac{ (1)_{m-1} (1-m)_{m-1}} {(\frac32)_{m-1} (m-1)!} \Big(-\frac1{b^2}\big)^{m-1}
\\& = \frac{\Gamma(\sigma-\frac12)\, \Gamma(\frac32) } { \Gamma(\sigma)} |b|^{1-2 \sigma}.
\end{align*}

The rest of the integral is estimated as follows. If $a\le 1$, the following term is absent; if $a>1$, we have
\begin{align*}
\int_0^{a-1} |u-z|^{-\Re s} \rho(u) du &\le C_n a^{\frac{n+1}2} (1+b^2)^{-\Re s/2}.
\end{align*}
The last term is
\begin{align*}
\int_{a+1}^\infty & |u-z|^{-\Re s} \rho(u) du 
 \le C_n \int_{a+1}^{2a+1} |1+ib|^{-\Re s} (2a+1)^{\frac{n-1}2} du
\\& \hskip2.5cm + C_n \int_{2a+1}^\infty \Big(\frac u2 \Big)^{-\Re s} u^{\frac{n-1}2} du
\\& \le C_n (2a+1)^{\frac{n+1}2} (1+b^2)^{-\Re s/2}
 + C_n 2^{\Re s} \frac{(a+1)^{\frac{n+1}2 -\Re s}} {\Re s-\frac{n+1}2}.
 \end{align*}
Note that for any $b$, $(1+b^2)^{-\Re s/2}\le |b|^{1-\Re s}$.
We conclude that (for $a\ge0$)
\begin{equation} \label{k-b-to-0}
\|k_{a+ib,s}\|_\infty \lesssim C_n (a+1)^{\frac{n-1}2} \dfrac{ \Gamma(\frac{\Re s-1}2)}{\Gamma(\frac{\Re s}2)} \exp(\frac\pi2 |\Im s|) |b|^{1-\Re s}, \qquad b\to0
\end{equation}

For $a,b$ fixed, we obtain also
$$
\|k_{a+ib,s}\|_\infty \lesssim C_n \exp(\frac\pi2 |\Im s|+\Re s\ln 2) \frac{(a+1)^{\frac{n+1}2 -\Re s}} {\Re s-\frac{n+1}2}, \qquad \Re s\downarrow\frac{n+1}{2}.
$$
For $\Im s\ne0$, $k_{a+ib,s}$ is actually bounded at $\Re s=\frac{n+1}{2}$, so one would expect that the last bound can be improved.

{\bf Case 2: $|z|\to\infty$}. Here we opt not to track dependence of the constants on $s$.

For $\Re z=a\ge0$ and $|b|\le1$ the analysis has been done above, with
$$
\|k_{a+ib,s}\|_\infty \le C_{n,s} \Big[ (a+1)^{\frac{n+1}2 -\Re s} + (a+1)^{\frac{n-1}2} |b|^{1-\Re s}\Big]
$$
(where clearly either term can be leading depending on $a$ and $b$).

For $\Re z=a\ge0$ and $|b|>1$, the correct power of $b$ is $-\Re s$; to verify this, it is sufficient to estimate
\begin{align*}
\int_0^{a+1} |u-z|^{-\Re s}\rho(u)du \le (a+1)^{\frac{n+1}2} |b|^{-\Re s},
\end{align*}
so that
$$
\|k_{a+ib,s}\|_\infty \le C_{n,s} (a+1)^{\frac{n+1}2} \Big[ (a+1)^{-\Re s} + |b|^{-\Re s}\Big].
$$
Finally, for $\Re z<0$ (and $|z|>1$) the estimates are straightforward:
\begin{align*}
\int_0^\infty |u-z|^{-\Re s}\rho(u)du & \le \int_0^{|z|} |z|^{-\Re s} \rho(u)du
 + \int_{|z|}^\infty u^{\frac{n-1}2 - \Re s} du
 \\& \le C_{n,s} |z|^{\frac{n+1}2 - \Re s}
\end{align*}
and 
$$
\|k_{z,s}\|_\infty \le C_{n,s} |z|^{\frac{n+1}2 - \Re s}.
$$

The norm $\|k_{z,s}\|_2$ can be estimated by the same bounds as $\|k_{z,2s}\|_\infty$, so that the formula \eqref{Lp-Lq} leads to the bound
$$
\|A_{z,s}\|_{\tilde L^p\to L^q} \le |z|^{\frac{n+1}2 - \Re s(1+\frac2{r'})} 
$$
in the case $\Re z<0$, $|z|\to\infty$, and this can be accordingly modified in the other cases.

\section{Lower norm estimates}\label{sec-lower}

Let $\psi\in C_0(\R)$ be a function which is not identically zero on $[0,+\infty)$ and twice differentiable with $\|\psi^{(j)}(s)s^k\|_\infty \le C_\psi$ for $0\le j\le 2$, $0\le k\le n/2+3$. These conditions are in particular verified for $\psi(s) = (1+s^2)^{-\alpha/2}$, $\alpha\ge n+3$.
In \cite{mueller-thiele}, M\"uller and Thiele have proved that the $L^1$-norm of the convolution kernel $k_t$ of $\psi(\sqrt{\cal L})\cos(t\sqrt{\cal L})$, as well as of $\psi(\sqrt{\cal L})\dfrac{\sin(t\sqrt{\cal L})}{\sqrt {\cal L}}$, is bounded by $C(1+|t|)$. We prove below that this estimate is sharp at $t\to+\infty$.
 
We thank W.~Hebisch for the following remark, made after the first version of the paper was completed. Let $f$ be such that the convolution kernel $k_{f,n}$ of $f(\cal L)$ is in $L^1(G_n)$, where $G_n$ is the $n$-dimensional `$ax+b$' group. Then $k_{f,n}$ is obtained from the kernel $k_{f,n+1}$ on the $n+1$-dimensional group by integration over the last coordinate \cite{hebisch-steger}. This trivially implies that $\|k_{f,n+1}\|_1\ge\|k_{f,n}\|_1$. In addition, the exact asymptotics is known in the case $n=2$ due to Hebisch, who used the transference principle by giving an isometry with functions of the Laplacian on $\R^3$.

This does not allow however to compare uniform norms in a similar way, and the case $n=1$ needs a full consideration. The remar
ks above could simplify our integration Lemmas \ref{int} and \ref{int-t} and a few details elsewhere; these changes seem not to
 be major and we decided to keep the proofs as they are.

The integration formulas over $G$ below are not identical to \cite{hebisch-edin} but are influenced by this article.

\subsection{The convolution kernel}

According to \eqref{k-psi} (which comes from \cite{mueller-thiele}), the convolution kernel $k_t$ of $\psi(\sqrt{\cal L}) \exp(it\sqrt{\cal L})$
$$
k_t (x,y)
 = c_l e^{-\frac n2 x} \int_0^\infty \psi(s) e^{its}
  \int_R^\infty D_{l,s}(v) (\ch v-\ch R)^{l-\frac n2} dv\,sds,
$$
where we denote as before $D_{l,s}(v) = D^l_{\sh, v}(e^{isv}-e^{-isv})$. Recall that $l$ is chosen as $l=\lfloor \frac n2\rfloor$.

Note first that
$$
\|k_t\|_2^2 \le \int_0^\infty |\psi(\sqrt u)|^2\rho(u)du
$$
is bounded uniformly in $t$, so that the same is true for $\int_{R(x,y)\le 1}|k_t|$. We can therefore assume in the sequel that $R\ge1$.

We start by transforming the kernel with the help of the following decomposition. It can be compared with \cite[Lemma 5.3]{mueller-thiele} where a variable-separating decomposition is obtained, but we need more precision on the terms appearing in it.

\begin{prop}\label{Dlu-polynomials}
$$
D^l_{\sh, v}(e^{iuv}) = \sum_{k=0}^l u^k e^{iuv} q_{k,l}(v),
$$
where $q_{k,l}(v) = P_{k,l}(\sh v, \ch v) (\sh v)^{-2l}$, every $P_{k,l}$ is a homogeneous polynomial of two variables of degree $l$, and $q_{k,l}$ is even for even $k$ and odd for $k$ odd. In particular, $D^l_{\sh ,u}(v)\to0$, $v\to+\infty$.
\end{prop}
\begin{proof}
The formula is trivially true for $l=0$. The induction step is proved by direct differentiation:
\begin{align*}
D^{l+1}_{\sh, v}(e^{iuv}) &= \frac d{dv} \sum_{k=0}^l u^k e^{iuv} \frac{q_{k,l}(v)}{\sh v}
\\& = \sum_{k=0}^l u^k e^{iuv}
 \Big[ iu \frac{q_{k,l}(v)}{\sh v} + \frac{q'_{k,l}(v)}{\sh v} - \frac{q_{k,l}(v)\ch v}{\sh^2 v} \Big]
\\& = e^{iuv}  \Big[ \frac{q'_{0,l}(v)}{\sh v} - \frac{q_{0,l}(v)\ch v}{\sh^2 v} \Big]
\\&+ \sum_{k=1}^l u^k e^{iuv}
 \Big[ i \frac{q_{k-1,l}(v)}{\sh v} + \frac{q'_{k,l}(v)}{\sh v} - \frac{q_{k,l}(v)\ch v}{\sh^2 v} \Big]
+ i u^{l+1} e^{iuv} \frac{q_{l,l}(v)}{\sh v}.
  \end{align*}
Set 
\begin{equation}\label{qkl}
q_{k,l+1}(v):=
\begin{cases}
\displaystyle \frac{q'_{0,l}(v)}{\sh v} - \frac{q_{0,l}(v)\ch v}{\sh^2 v} & k=0\\
\displaystyle i \frac{q_{k-1,l}(v)}{\sh v} + \frac{q'_{k,l}(v)}{\sh v} - \frac{q_{k,l}(v)\ch v}{\sh^2 v}  & 1\le k \le l\\
\displaystyle i \frac{q_{l,l}(v)}{\sh v} &k=l+1.
\end{cases}
\end{equation}
Then 
\begin{equation*}
D^{l+1}_{\sh, v}(e^{iuv}) = \sum_{k=0}^{l+1} u^k e^{iuv} q_{k,l+1}(v).
\end{equation*}
It remains to check that there exist homogeneous polynomials of two variables $P_{k,l+1},0\le k\le l+1$ which are of degree $l+1$ such that 
\begin{equation*}
P_{k,l+1}(\sh v,\ch v):=(\sh v)^{2l+2}q_{k,l+1}(v).
\end{equation*}
For this, we claim that $q_{k,l}'(v)=(\sh v)^{-2l-1}Q_{k,l}(\sh v,\ch v)$ with $Q_{k,l}$ being a two-variable homogeneous polynomial of degree $l+1$. In fact, 
\begin{align*}
q_{k,l}'(v)=&(\sh v)^{-2l}( P_{k,l,1}(\sh v,\ch v)\ch v-P_{k,l,2}(\sh v,\ch v)\sh v)\\
&-2l(\sh v)^{-2l-1}\ch v P_{k,l}(\sh v,\ch v),
\end{align*}
where $P_{k,l,1}(x,y):=\partial_x P_{k,l}(x,y)$ and $P_{k,l,2}(x,y):=\partial_y P_{k,l}(x,y)$ are two-variable homogeneous polynomials of degree $l-1$. So 
\begin{equation*}
(\sh v)^{2l+1}q_{k,l}'(v)=Q_{k,l}(\sh v,\ch v),
\end{equation*}
where 
\begin{equation}\label{eq:Q_k,l}
Q_{k,l}(x,y):=P_{k,l,1}(x,y)xy-P_{k,l,2}(x,y)x^2-2lP_{k,l}(x,y)y
\end{equation}
is a two-variable homogeneous polynomial of degree $l+1$. This completes the proof of the claim.

When $k=0$, by induction and by our claim:
\begin{align*}
(\sh v)^{2l+2}q_{0,l+1}(v)
&=(\sh v)^{2l+1}q'_{0,l}(v)-(\sh v)^{2l}\ch(v)q_{0,l}(v)\\
&=Q_{0,l}(\sh v,\ch v)-\ch(v)P_{0,l}(\sh v,\ch v).
\end{align*}
so we can choose $P_{0,l+1}(x,y)=Q_{0,l}(x,y)-yP_{0,l}(x,y)$.

When $1\le k\le l$,
\begin{align*}
(\sh v)^{2l+2}q_{k,l+1}(v)
&=i (\sh v)^{2l+1}q_{k-1,l}(v)+ (\sh v)^{2l+1} q'_{k,l}(v) - (\sh v)^{2l}q_{k,l}(v)\ch v \\
&=i P_{k-1,l}(\sh v,\ch v)\sh v+ Q_{k,l}(\sh v, \ch v) - P_{k,l}(\sh v,\ch v)\ch v 
\end{align*}
so we can choose $P_{k,l+1}(x,y)=i xP_{k-1,l}(x,y)+Q_{k,l}(x,y)-yP_{k,l}(x,y)$.

When $k=l+1$, by induction:
\begin{align*}
(\sh v)^{2l+2}q_{l+1,l+1}(v)
&=i(\sh v)^{2l+1}q_{l,l}(v)=iP_{l,l}(\sh v,\ch v)\sh v,
\end{align*}
so we can choose $P_{l+1,l+1}(x,y)=iP_{l,l}(x,y)x$.

Finally, statements on parity of $q_{k,l}$ are easy to check.
\end{proof}

Note that we do not check whether $P_{k,l}$ are nonzero, so with a certain abuse of language we assume zero to be a homogeneous polynomial of any degree. All we need to know of $q_{k,l}$ in this respect is contained in the Proposition below:

\begin{prop}\label{lemma-akl}
There exist constants $b_l$ and $a_{k,l}$, $0\le k\le l$, such that
$$
|q_{k,l}(v)- a_{k,l} e^{-vl}|\le b_{l} e^{-2vl}
$$
for every $v\in[0,+\infty)$. Moreover, $a_{l,l}\ne0$.
\end{prop}
\begin{proof}
By Proposition \ref{Dlu-polynomials}, we have
$$
(\sh v)^{2l}q_{k,l}(v) = P_{k,l}(\sh v, \ch v),
$$
for some two-variable homogeneous polynomial $P_{k,l}$ of degree $l$. It is easy to see that there exists a polynomial 
$A_l$ of degree $4l$ such that 
$$
e^{2vl}(\sh v)^{2l}=(e^v\sh v)^{2l}=A_l(e^v),
$$
and a polynomial $B_{k,l}$ of degree $2l$ such that 
$$
e^{vl}P_{k,l}(\sh v, \ch v)=P_{k,l}(e^v\sh v, e^v\ch v)=B_{k,l}(e^v).
$$
Thus there exist some nonzero constant $a_{k,l}$ and some polynomial $C_{k,l}$ of degree $\le 4l-1$ such that 
$$
e^{vl}q_{k,l}(v) = \frac{e^{vl} P_{k,l}(\sh v, \ch v)}{(\sh v)^{2l}} = \frac{e^{2vl}B_{k,l}(e^v)}{A_l(e^v)}=a_{k,l}+\frac{C_{k,l}(e^v)}{A_l(e^v)}.
$$
Hence for some constant $b_{l}>0$ we have
\begin{equation}\label{akl-diff}
|e^{vl}q_{k,l}(v)- a_{k,l} |\le b_{l}e^{-vl},
\end{equation}
or equivalently
$$|q_{k,l}(v)- a_{k,l} e^{-vl}|\le b_{l} e^{-2vl}.$$

It remains to show that $a_{l,l}\ne0$. The function $q_{l,l}$ is in fact easy to calculate: by \eqref{qkl}, $q_{l,l}(v) = \Big(\dfrac i{\sh v} \Big)^l$.
For $l=0$, we have $a_{0,0}=1$. For $l>0$, by \eqref{akl-diff},
$$
a_{k,l} = \lim_{v\to+\infty} e^{vl}q_{k,l}(v),
$$
so that $a_{l,l} = (2i)^l\ne0$.
\end{proof}

We can write now
\begin{align}\label{kt-formula}
&k_t(x,y) = c_l e^{-\frac{nx}2} \int_R^\infty (\ch v-\ch R)^{l-\frac n2}
 \\&\hskip4cm \int_0^\infty \sum_{k=0}^l [s^k e^{isv} - (-s)^k e^{-isv}]\, q_{k,l}(v)\psi(s) \,s\,e^{its}dv\,ds \notag
\\ &= \frac{c_l}2 e^{-\frac{nx}2}  \int_R^\infty (\ch v-\ch R)^{l-\frac n2} 
 \sum_{k=0}^l\, q_{k,l}(v) [\,\check m_k(t+v)  - (-1)^k \check m_k(t-v)]\, dv \notag
\end{align}
where $m_k(s) = \psi(s) s^{k+1} I_{[0,+\infty)}$ and $\check{m}_k$ is the inverse Fourier transform, which we write without additional constants.

While $m_k(0)=0$ for all $k$, the derivatives of $m_k$ may be discontinuous at $0$, so we need some attention when estimating $\check m_k$. Integrating by parts, we can write, for $0\ne\xi\in\R$:
\begin{align}\label{mk-by-parts}
\check m_k(\xi) 
& = \frac1{\xi^2} \big[m_k'(s) e^{i\xi s}\big]_0^\infty - \frac1{\xi^2} \int_0^\infty m_k''(s) e^{i\xi s}ds,
\end{align}
which is bounded as
\begin{equation}\label{mk}
|\check m_k(\xi)| \le C_{l,\psi} |\xi|^{-2}
\end{equation}
for $0\le k\le l$.
One can note also that $m_k\in L^1(\R)$ for every $k$ and $|\check m_k(\xi)|\le C_{l,\psi}$.

We will separate \eqref{kt-formula} into two parts: $k_{t}(x,y)=\frac{c_l}2 (I_1-I_2)$ with 
\begin{equation}\label{eq-I1}
I_1=  e^{-\frac{nx}2} \int_R^\infty (\ch v-\ch R)^{l-\frac n2} \sum_{k=0}^l\, q_{k,l}(v) \check m_k(t+v) dv
\end{equation}
and 
\begin{equation}\label{eq-I2}
I_2= e^{-\frac{nx}2}  \int_R^\infty (\ch v-\ch R)^{l-\frac n2} 
\sum_{k=0}^l\, q_{k,l}(v) (-1)^{k+1} \check m_k(t-v)\, dv.
\end{equation}
Bounds for $\|k_t\|_1$ will be derived from estimates of the integrals of $I_1$ and $I_2$ over a certain set $\{(x,y)\in G: R(x,y)\in [t+a,t+b]\}$. These estimates are eventually reduced to the following Lemmas \ref{int} and \ref{int-t}, preceded by a short calculation in Lemma \ref{beta}.

\subsection{Integration lemmas}

\begin{lm}\label{beta}
	For $\alpha>-1$ and $\beta>\alpha $, we have 
	\begin{equation*}
	\int_{0}^{\infty}(e^v-1)^\alpha e^{-\beta v}dv=B(\alpha+1,\beta -\alpha)<\infty,
	\end{equation*}
	where $B(\cdot,\cdot)$ is the beta function.
\end{lm}

\begin{proof}
	This is a direct computation:
	\begin{align*}
	\int_{0}^{\infty}(e^v-1)^\alpha e^{-\beta v}dv
	=&\int_{0}^{\infty}(1-e^{-v})^\alpha e^{-(\beta-\alpha-1)v}e^{-v}dv\\
	=&\int_{0}^{1}(1-x)^\alpha x^{\beta-\alpha-1}dx\\
	=&B(\alpha+1,\beta -\alpha)<\infty.
	\end{align*}
\end{proof}

\begin{lm}\label{int}
For any $a<b$, $m\ge0$ and $t>\max(|a|,b)$ there exists a positive constant $C_{n,a,b}$ such that
$$
\int_{R\in[t+a,t+b]} e^{-n\frac{x+R}2} (t+R)^{-m} dx dy\le C_{n,a,b}\, t^{1-m}.
$$
\end{lm}

\begin{proof}
As the integrand depends on $x$ and $|y|$ only, we can pass to ($n$-dimensional) polar coordinates in $y = (r,\Phi)$. For $x,r\in \R$, denote $R'(x,r) = R(x,(r,0))$, and $F_{a,b} = \{(x,r) \in \R^2 : R'(x,r) \in [t + a,t + b]\}$. We obtain
	\begin{align*}
I_{a,b} &=	\int_{R\in[t+a,t+b]} e^{-n\frac{x+R}2} (t+R)^{-m} dx dy	
	\\&=V_n \int_{F_{a,b}} e^{-n\frac{x+R'}2} (t+R')^{-m}r^{n-1}  dx dr,
	\end{align*}
	where $V_n$ is the volume of the unit ball in $\R^n$. By \eqref{R} and inequalities just after it, the integral can be bounded by
\begin{align*}
I_{a,b}
&\le 2^{\frac{n-1}2} V_n \int_{F_{a,b}} e^{-\frac{x+R'}2} t^{-m}  dx dr
\\&\le 2^{\frac{n-1}2} V_n \int_{F_{a,b}} e^{-\frac{x+t+a}2} t^{-m}  dx dr.
\end{align*}
Again by \eqref{R}, if $(x,r)\in F_{a,b}$, then $|x|\le R'(x,r)\le t+b$, and $0\le r\le \sqrt 2\exp\big(\frac{x+R'}2\big) \le 2\exp\big(\frac{x+t+b}2\big)$.
This implies that
\begin{align*}
\int_{F_{a,b}} e^{-\frac{x+t+a}2} t^{-m}  dx dr &\le t^{-m} \int_{|x|\le t+b} e^{-\frac{x+t+a}2} \int_0^{2\exp\big(\frac{x+t+b}2\big)} dr dx
\\& = 2 t^{-m} \int_{|x|\le t+b} e^{\frac{b-a}2} dx = 4 e^{\frac{b-a}2} t^{-m} (t+b)
\\&\le 8 e^{\frac{b-a}2} t^{1-m}
\end{align*}
(since we suppose $b<t$).
\end{proof}

\begin{lm}\label{int-t}
For $a<b$ and $t>\max(1,2|a|,b)$, there exists a constant $C'_{n,a,b}>0$ such that
$$
\int_{t+a\le R\le t+b} e^{-n\frac{x+R}2} dx dy \ge C'_{n,a,b}\, t.
$$
\end{lm}
\begin{proof}
Denote $G_{a,b} = \{(x,y)\in G: R(x,y)\in [t+a,t+b]\}$. From \eqref{R}, we have
$$
|y|^2 = 2 e^x (\ch R - \ch x).
$$ 
This formula implies that for $(x,y)$ such that $|x|\le t+a$ and
$$
2e^x (\ch (t+a)-\ch x) \le |y|^2 \le 2e^x (\ch (t+b)-\ch x),
$$
$R(x,y)$ is between $t+a$ and $t+b$, so that $G_{a,b}$ contains for each $x$ a ``thick sphere'' with $|y|$ changing according to the bounds above (under condition $|x|\le t+a$).
Again denoting by $V_n$ the volume of the unit ball in $\R^n$, we can estimate the volume of this ``sphere'' as follows:
\begin{align*}
V & = V_n 2^{\frac n2} e^{n\frac x2}
 \big[ \big(\ch (t+b)-\ch x\big)^{n/2} - \big(\ch (t+a)-\ch x\big)^{n/2} \big]
\\& = c_n e^{n\frac x2} \frac{ \big(\ch (t+b)-\ch x\big)^n - \big(\ch (t+a)-\ch x\big)^n} {\big(\ch (t+b)-\ch x\big)^{n/2} + \big(\ch (t+a)-\ch x\big)^{n/2}}
\\& \ge c_n e^{n\frac x2} \frac{(\ch(t+b)-\ch(t+a))
\,n \big(\ch (t+a)-\ch x\big)^{n-1} } {2\big(\ch (t+b)\big)^{n/2}}
\\& \ge \frac{c_n}2 e^{n\frac x2} e^{-\frac n2(t+b)} \big(\ch(t+b)-\ch(t+a)\big) \big(\ch (t+a)-\ch x\big)^{n-1}.
\end{align*}
Since $-2t-a-b=-2(t+a)+a-b\le a-b<0$, we can continue as
\begin{align*}
\ch(t+b)-\ch(t+a) = \frac12 e^t (e^b-e^a) (1-e^{-2t-a-b}) \ge \frac12 e^t (e^b-e^a)(1-e^{a-b}),
\end{align*}
and
$$
V \ge c_{n,a,b}\, e^{n\frac x2} e^{(1-\frac n2)t}\big(\ch (t+a)-\ch x\big)^{n-1}.
$$

The integral thus can be bounded by:
\begin{align*}
\int_{G_{a,b}} e^{-n\frac{x+R}2} dx dy 
&\ge c_{n,a,b}\, e^{(1-\frac n2)t} \int_{|x|\le t+a} e^{n\frac x2} e^{-n\frac{x+t+b}2} \big(\ch (t+a)-\ch x\big)^{n-1} dx
\\&= c'_{n,a,b}\, e^{(1-n)t} \int_0^{t+a} \big(\ch (t+a)-\ch x\big)^{n-1} dx
\\&= c'_{n,a,b}\, e^{(1-n)t} \int_0^{t+a} \Big[ \frac12 e^{t+a} (1-e^{x-t-a}) (1-e^{-x-t-a})\Big]^{n-1} dx
\\&\ge c''_{n,a,b} \int_0^{t+a} (1-e^{x-t-a})^{n-1} dx.
\end{align*}
By assumption $t>|a|$, so that $t/4<t+a$. If $0\le x\le t/4$, then $x-t-a \le (-t/2-a)-t/4\le -t/4$, so that
$$
\int_{G_{a,b}} e^{-n\frac{x+R}2} dx dy \ge c''_{n,a,b} \int_0^{t/4} (1-e^{-t/4})^{n-1} dx \ge c'''_{a,b,n}\, t,
$$
where in the last inequality we use the assumption $t>1$.
\end{proof}

\subsection{Estimates of $I_1$ and $I_2$}\label{sec-I1-I2}

\subsubsection{Estimates of $I_1$}

We are interested in $R\ge 1$, so that in \eqref{eq-I1} we have $v\ge1$. By Lemma~\ref{Dlu-polynomials}, $|q_{k,l}(v)|\le C_l e^{-lv}$ with some constant $C_l$.

Together with \eqref{mk}, we have
\begin{align*}
|I_1|&\le e^{-\frac{nx}2} \int_R^\infty (\ch v-\ch R)^{l-\frac n2} 
 \sum_{k=0}^l\, \Big| q_{k,l}(v) \check m_k(t+v) \Big|dv
\\&\le C_{l,\psi} e^{-\frac{nx}2} \int_R^\infty (\ch v-\ch R)^{l-\frac n2} e^{-lv} (t+v)^{-2} dv
\\&\le C_{l,\psi} e^{-\frac{nx}2} (t+R)^{-2} \int_0^\infty (\ch (v+R)-\ch R)^{l-\frac n2} e^{-lv-lR} dv.
 \end{align*}
One can transform
$$
2(\ch(v+R) - \ch R) 
 = (e^v-1)(e^R-e^{-v-R}) = (e^v-1)e^R (1-e^{-v-2R}),
$$
and estimate
\begin{align*}
|I_1| &\le C_{l,\psi} e^{-\frac{nx}2} (t+R)^{-2} \int_0^\infty ((e^v-1)e^R)^{l-\frac n2} e^{-lv-lR} dv
 \\&= C_{l,\psi} e^{-\frac{n(x+R)}2} (t+R)^{-2} \int_0^\infty (e^v-1)^{l-\frac n2} e^{-lv} dv.
\end{align*}
By Lemma \ref{beta}, the integral converges (and depends only on $n$), so that
\begin{equation}\label{I1-ptwise}
|I_1| \le C_{n,\psi} e^{-\frac{n(x+R)}2} (t+R)^{-2}.
\end{equation}
Now by Lemma \ref{int}, for $a<b$ and $t>\max(|a|,b)$
\begin{equation}\label{I1}
\int_{R\in[t+a,t+b]} |I_1| dxdy
 \le C_{l,\psi} t^{-1}.
\end{equation}

\subsubsection{Simplifications of $I_2$}

We can now pass to the second term:
\begin{align*}
I_2 & = e^{-\frac{nx}2}  \int_R^\infty (\ch v-\ch R)^{l-\frac n2} 
 \sum_{k=0}^l\, q_{k,l}(v) (-1)^{k+1} \check m_k(t-v)\, dv
\\ &= e^{-\frac{nx}2}  \int_0^\infty (\ch (v\!+\!R)-\ch R)^{l-\frac n2} 
 \sum_{k=0}^l\, q_{k,l}(v\!+\!R) (-1)^{k+1} \check m_k(t\!-\!R\!-\!v)\, dv.
\end{align*}
Its analysis will pass through several stages of simplification. By Lemma \ref{lemma-akl},
$$
|q_{k,l}(v)- a_{k,l} e^{-vl}|\le b_{l} e^{-2vl}.
$$
If we replace $q_{k,l}(v+R)$ by $a_{k,l} e^{-(v+R)l}$ in $I_2$ and denote
\begin{align*}
I_3 & := e^{-\frac{nx}2} \int_0^\infty (\ch (v+R)-\ch R)^{l-\frac n2} 
 \sum_{k=0}^l\, a_{k,l} e^{-(v+R)l} (-1)^{k+1} \check m_k(t\!-\!R\!-\!v)\, dv,
\end{align*}
then it is sufficient to estimate $I_3$ instead of $I_2$ because their difference is bounded as follows:
\begin{align*}
|I_2-I_3| &
\le e^{-\frac{nx}2}  \int_0^\infty (\ch (v+R)-\ch R)^{l-\frac n2} 
\\&\hskip3cm \sum_{k=0}^l\Big|q_{k,l}(v+R)-a_{k,l}e^{-(v+R)l}\Big|  |\check m_k(t-R-v)|\, dv
\\ 
& \le c'_l e^{-\frac{nx}2}  \int_0^\infty (\ch (v+R)-\ch R)^{l-\frac n2} 
 e^{-2(v+R)l} \sum_{k=0}^l |\check m_k(t-R-v)|\, dv.
\end{align*}
If $l=0$, then $I_3$ is exactly equal to $I_2$ since $q_{0,0}\equiv a_{0,0}=1$. For $l>0$, we need to show that $|I_3-I_2|$ is small enough.
By definition,
\begin{align*}
 (\ch (v+R)-\ch R)^{l-\frac n2} 
 &=2^{\frac{n}{2}-l}(e^v-1)^{l-\frac{n}{2}}e^{R(l-\frac{n}{2})}(1-e^{-v-2R})^{l-\frac{n}{2}}.
\end{align*}
Since $l-\frac{n}{2}\le0$ and $R\ge1$, 
$$
(1-e^{-v-2R})^{l-\frac{n}{2}} \le (1-e^{-2})^{l-\frac{n}{2}}.
$$
Next, since every $m_k$ is in $L^1$, we have $\sum_{k=0}^l |\check m_k(t-R-v)|\le C_{n,\psi}$. These, together with Lemma \ref{beta}, yield
\begin{align*}
|I_2-I_3| & \le C_{n,\psi,l}\, e^{-\frac{nx}2}  \int_0^\infty (e^v-1)^{l-\frac n2} 
 e^{-2(v+R)l+R(l-\frac{n}{2})} dv
\\ & \le C_{n,\psi,l}\, e^{-\frac{n(x+R)}2-Rl} \int_0^\infty (e^v-1)^{l-\frac n2} e^{-2vl}dv
\\ & = C_{n,\psi,l}\,  B\big(l-\frac{n}{2}+1,l+\frac{n}{2}\big)e^{-\frac{n(x+R)}2-Rl},
\end{align*}
which has, by Lemma \ref{int}, the integral bounded as follows (with the same assumptions on $t$ as in the lemma):
\begin{align}\label{int-I2-I3}
\int_{t+a\le R \le t+b} |I_2-I_3| dxdy &\le C'_{n,\psi} \, e^{-(t+a)l} \int_{t+a\le R\le t+b} e^{-\frac{n(x+R)}2} dxdy \notag
\\ &\le C_{n,\psi} \, e^{-tl} t.
\end{align}

Our next step is to pass to
\begin{align*}
I_4 & := e^{-\frac{nx}2} \int_0^\infty\! 2^{\frac{n}{2}-l}e^{R(l-n/2)}(e^v\!-\!1)^{l-\frac n2} 
 \sum_{k=0}^l a_{k,l} e^{-(v+R)l} (-1)^{k+1} \check m_k(t\!-\!R\!-\!v)\, dv
\end{align*}
showing that $I_3-I_4$ is small.
If $l=\frac n2$, then $I_3=I_4$ is an exact equality. If $l-\frac{n}{2}=-\frac12$, we have to estimate
\begin{align*}
D &:= \Big|(\ch (v+R)-\ch R)^{l-\frac n2} -2^{\frac{n}{2}-l}(e^v-1)^{l-\frac{n}{2}}e^{R(l-\frac{n}{2})}\Big|\\
&=2^{\frac{n}{2}-l}(e^v-1)^{l-\frac{n}{2}}e^{R(l-\frac{n}{2})}\Big|1-(1-e^{-v-2R})^{l-\frac{n}{2}}\Big|.
\end{align*}
Set $z=e^{-v-2R}$. By the choice of $R$ and $v$, we have $z\in (0,e^{-2}]$. In this interval, the function $f(z)=(1-z)^{-1/2}$ has a derivative $f'(z) = \frac12 (1-z)^{-3/2}$ bounded by $f'(e^{-2})$, so that
$$
\Big|1-(1-e^{-v-2R})^{l-\frac{n}{2}}\Big| = |f(0)-f(z)| \le f'(e^{-2})\, z = f'(e^{-2})\,e^{-v-2R}.
$$
This implies that
$$
D \le C_n (e^v-1)^{l-\frac{n}{2}}e^{R(l-\frac{n}{2})} e^{-v-2R}
$$
and
\begin{align*}
|I_3-I_4| & \le C_n e^{-\frac{nx}2} \int_0^\infty (e^v-1)^{l-\frac{n}{2}}e^{R(l-\frac{n}{2})} e^{-v-2R} e^{-(v+R)l}
\\&\hskip4.5cm \sum_{k=0}^{l}|a_{k,l}||\check m_k(t-R-v)|dv
\\ & \le C_{n,\psi,l}\, e^{-\frac{n(x+R)}2-2R} \int_0^\infty (e^v-1)^{l-\frac n2} e^{-(l+1)v}\, dv
\\ &= C_{n,\psi}B\big(l-\frac{n}{2}+1,\frac{n}{2}+1\big)\, e^{-\frac{n(x+R)}2-2R},
\end{align*}
with the integral bounded by
\begin{align}\label{int-I3-I4}
\int_{t+a\le R \le t+b} |I_3-I_4| dxdy &\le C'_{n,\psi} \, e^{-2(t+a)} \int_{t+a\le R\le t+b} e^{-\frac{n(x+R)}2} dxdy \notag
\\ &\le C_{n,\psi} \, e^{-2t} t.
\end{align}

\subsubsection{Estimates of $I_4$}\label{sec-I4}

It remains now to estimate $I_4$ from below. Recall that 
\begin{align*}
I_4 & = c'_n e^{-\frac{n(x+R)}2} \int_0^\infty (e^v-1)^{l-\frac n2} 
 \sum_{k=0}^l\, a_{k,l} e^{-vl} (-1)^{k+1} \check m_k(t-R-v)\, dv.
\end{align*}
Set
$$
\cal M(\xi) = \sum_{k=0}^l\, a_{k,l} (-1)^{k+1} \check m_k(\xi);
$$
this is the inverse Fourier transform of
$$
\Psi(s) = \sum_{k=0}^l\, a_{k,l} (-1)^{k+1} s^{k+1} \psi (s) I_{[0,+\infty)}.
$$
By assumption, $\psi$ is not identically zero on $[0,+\infty)$; it is multiplied by a nonzero polynomial since $a_{l,l}\ne0$ by Lemma \ref{lemma-akl}, so finally the product $\Psi$ is not identically zero.
Consider the integral entering in $I_4$,
\begin{equation}\label{Ixi}
\cal I(\xi) = \int_0^\infty (e^v-1)^{l-\frac n2} e^{-vl} \cal M(\xi-v)\, dv
\end{equation}
as the convolution of $\cal M$ and $\cal N(v):= (e^v-1)^{l-\frac n2} e^{-vl}1_{(0,\infty)}(v)$. The function $\cal I$ is continuous, and its inverse Fourier transform is $\Psi \check {\cal N}$.
For every $m$, we can estimate
\begin{align*}
\int_\R |\cal N(v)v^m|dv &\le \int_0^1 (e^v-1)^{l-\frac n2}dv + \int_1^\infty C_n e^{-v\frac n2} v^m dv
\\& \le C'_n + C_n \Big( \frac 2n\Big)^{m+1} \int_{n/2}^\infty e^{-x} x^m dx
\le C''_n K^m m!
\end{align*}
with $K=\max(1,2/n)$. It follows that $\|\check {\cal N}^{(m)}\|_\infty \le C_n K^m m!$, so by \cite[19.9]{rudin} $\check {\cal N}$ is analytic in a strip $|z|<\delta$ with some $\delta>0$. As it is clearly not identically zero, we can conclude that $\Psi \check {\cal N}$ is not identically zero too.

There exist therefore $\a,\beta\in\R$, $\alpha<\beta$, $A\ne0$ and $0<\e<|A|/2$ (all these constants depend on $\psi$ and on $n$) such that $|\cal I(\xi)-A|<\e$ for $\xi\in[\a,\beta]$. It follows that for $R\in[t-\beta,t-\a]$
\begin{equation}\label{I4-ptwise}
|I_4| > C_{n,\psi} e^{-\frac{n(x+R)}2}
\end{equation}
and by Lemma \ref{int-t}, with $a=-\beta$, $b=-\a$,
\begin{equation}\label{I4}
\int_{R\in[t+a,t+b]} |I_4|dxdy \ge C_{n,\psi}\,t,
\end{equation}
once $t>\max(1,2|a|,b)$.

\subsection{Conclusions}\label{sec-conclusions}

We can now summarize the obtained estimates in the following theorem.

\begin{thm}\label{th-L1}
Let $\psi\in C_0(\R)$ be a function which is not identically zero on $[0,+\infty)$ and twice differentiable with $\|\psi^{(j)}(s)s^k\|_\infty \le C_\psi$ for $0\le j\le 2$, $0\le k\le n/2+3$. Let $k_t$ be the convolution kernel of one of the following operators:
\begin{align*}
E_t &= \psi(\sqrt{\cal L})\exp(it\sqrt{\cal L}),\\
C_t &= \psi(\sqrt{\cal L})\cos(t\sqrt{\cal L}),\\
S_t &= \psi(\sqrt{\cal L}) \dfrac{\sin(t\sqrt{\cal L})}{\sqrt {\cal L}}.
\end{align*}
Then there exists a constant $C_{n,\psi}$ depending on $n$ and on $\psi$ such that for $t>0$ sufficiently large,
$$
\|k_t\|_1 \ge C_{n,\psi}\,t.
$$
\end{thm}
\begin{proof}
{\bf Case 1:} $k_t$ is the kernel of $E_t$. By the results of Section \ref{sec-I4}, there exist $a,b\in\R$ such that \eqref{I4} holds, for $t>\max(1,2|a|,b)$. For the same $a,b,t$ we have, with \eqref{I1}:
$$
\|k_t\|_1 \ge \frac{c_l}2 \int_{R\in[t+a,t+b]}|I_1-I_2|dxdy
\ge C_{n,\psi} \Big(\int_{R\in[t+a,t+b]}|I_2|dxdy -\frac1t \Big).
$$
Set $\gamma_0=0$ and $\gamma_l=1$ if $l\ne0$. We can continue, using the results above, as
\begin{align*}
\|k_t\|_1 & \overset{\eqref{int-I2-I3}}{\ge} C_{n,\psi} \Big(\int_{R\in[t+a,t+b]}|I_3|dxdy - \gamma_l\, e^{-lt} t - \frac1t \Big)
\\& \overset{\eqref{int-I3-I4}}{\ge} C_{n,\psi} \Big(\int_{R\in[t+a,t+b]}|I_4|dxdy - e^{-2t} t - \gamma_l\, e^{-lt} t - \frac1t \Big)
\\& \overset{\eqref{I4}}{\ge} C_{n,\psi} \Big(t - e^{-2t} t - \gamma_l\, e^{-lt} t - \frac1t \Big).
\end{align*}
For $t$ big enough, this is clearly bounded from below by $C_{n,\psi}\,t$, as claimed.

{\bf Case 2:} $k_t$ is the kernel of $C_t$. We represent $\cos(its)$ as $\frac12(e^{its}+e^{-its})$; in the integral \eqref{kt-formula}, instead of $[\,\check m_k(t+v)  - (-1)^k \check m_k(t-v)]$ we then obtain
$$
\frac12\big[\check m_k(t+v)  + \check m_k(v-t) - (-1)^k \check m_k(t-v) - (-1)^k \check m_k(-t-v)\big].
$$
We separate then $k_t=\frac{c_l}4(I_1-I_2)$ into
\begin{align*}
I_1 &=  e^{-\frac{nx}2} \int_R^\infty (\ch v-\ch R)^{l-\frac n2} \sum_{k=0}^l\, q_{k,l}(v)
[ \check m_k(t+v) - (-1)^k \check m_k(-t-v) ] dv,
\\I_2 &= e^{-\frac{nx}2}  \int_R^\infty (\ch v-\ch R)^{l-\frac n2} 
\sum_{k=0}^l\, q_{k,l}(v) [ (-1)^{k+1} \check m_k(t-v)+ \check m_k(v-t)]\, dv.
\end{align*}
As $|\check m_k(-t-v)| \le C_{n,\psi} (t+v)^{-2}$, the estimate for $I_1$ remains the same. The passage from $I_2$ to $I_4$ does not change either. In $I_4$ we still have \eqref{Ixi}, but with
$$
\cal M(\xi) = \sum_{k=0}^l\, a_{k,l} \big[\check m_k(-\xi) + (-1)^{k+1} \check m_k(\xi)\big].
$$
If we denote by $\tilde m_k$ the function $\tilde m_k(s) = m_k(-s)$, then $\cal M$ is the inverse Fourier transform of
$$
\Psi(s) = \sum_{k=0}^l\, a_{k,l} \big[ \tilde m_k(s) + (-1)^{k+1} m_k(s) \big]
$$
which is for $s\ge0$ the same as in Case 1, so that it is not identically zero for similar reasons as above. This implies that $\cal I$ is not everywhere zero, so that we arrive at the same conclusion:
$\|k_t\|_1 \ge C_{\psi,n}\, t$, for $t$ big enough.

{\bf Case 3:} $k_t$ is the kernel of $S_t$. In \eqref{kt-formula}, we obtain
$$
\check m_k(t+v)  - \check m_k(v-t) - (-1)^k \check m_k(t-v) + (-1)^k \check m_k(-t-v),
$$
but $m_k$ changes to $m_k(s) = s^k \psi(s) I_{[0,+\infty)}$ and not with $s^{k+1}$. As a consequence, we cannot integrate by parts twice as in \eqref{mk-by-parts} but only once; this changes the estimate \eqref{mk} as to
\begin{equation}\label{mk-sin}
|\check m_k(\xi)|\le C_{n,\psi}|\xi|^{-1}.
\end{equation}
Let us write down the two parts of $k_t$, up to the constant $c_l/4$:
\begin{align*}
I_1 &=  e^{-\frac{nx}2} \int_R^\infty (\ch v-\ch R)^{l-\frac n2} \sum_{k=0}^l\, q_{k,l}(v)
[ \check m_k(t+v) + (-1)^k \check m_k(-t-v) ] dv,
\\I_2 &= e^{-\frac{nx}2}  \int_R^\infty (\ch v-\ch R)^{l-\frac n2} 
\sum_{k=0}^l\, q_{k,l}(v) [ (-1)^{k+1} \check m_k(t-v)- \check m_k(v-t)]\, dv.
\end{align*}
With \eqref{mk-sin} instead of \eqref{mk}, we obtain
$$
\int_{R\in[t+a,t+b]} |I_1| dx dy \le C_{n,\psi}.
$$
The estimates of $|I_2-I_3|$ and $|I_3-I_4|$ do not change since we are using only the fact $m_k\in L^1(\R)$ which remains true.
In $I_4$ we get \eqref{Ixi} with
$$
\cal M(\xi) = \sum_{k=0}^l\, a_{k,l} \big[-\check m_k(-\xi) + (-1)^{k+1} \check m_k(\xi)\big],
$$
and complete the proof as in Case 2.
\end{proof}

It is well known that $L^1$-norm of a function $f$ is also the norm of the convolution operator $g\mapsto g*f$ on $L^1(G)$. We get as a corollary that upper norm estimates of \cite{mueller-thiele} for $C_t$, $S_t$ are sharp:
$$
\|C_t\|_{L^1\to L^1} \asymp t, \quad \|S_t\|_{L^1\to L^1} \asymp t,
$$
as $t\to\infty$.
Our results are valid in particular for $\psi(s) = (1+s^2)^{-\alpha}$, $\a>n/2+3$.

\subsection{Lower bounds for uniform norms}\label{sec-uniform}

In \cite[Corollary 7.1]{mueller-thiele}, M\"uller and Thiele show that for a function $\psi\in C^\infty_0(\R)$ supported in $[a,b]$, the kernel $k_t$ of $C_t = \psi(\sqrt{\cal L}) \cos(t\sqrt{\cal L})$ has its uniform norm bounded by a constant, with no decay in $t$ as $t\to+\infty$. With a simple modification, this result can be extended to functions which are not compactly supported, but decaying quickly enough.

From the results obtained in Sections \ref{sec-I1-I2}-\ref{sec-conclusions}, it follows that this estimate is actually sharp.

\begin{thm}
In the assumptions and notations of Theorem \ref{th-L1}, there exists a constant $C_{n,\psi}$ depending on $n$ and on $\psi$ such that for $t>0$ sufficiently large,
$$
\|k_t\|_\infty \ge C_{n,\psi}.
$$
\end{thm}
\begin{proof}
We can consider the cases of $E_t$ and $C_t$ together.

{\bf Case 1:} $k_t$ is the kernel of $E_t$ or $C_t$. By the results of Section \ref{sec-I4}, there exist $a,b\in\R$ such that \eqref{I4-ptwise} holds, for $t>\max(1,2|a|,b)$. It is clear then that we should consider $x=-R$, $y=0$. Arguing similarly to Theorem \ref{th-L1}, we obtain the estimate
$$
|k_t(-R,0)| \ge C_{n,\psi} \Big(1 - e^{-2t} - \gamma_l\, e^{-lt} - \frac1{t^2} \Big),
$$
which is bounded from below by $C_{n,\psi}$, as claimed.

{\bf Case 2:} $k_t$ is the kernel of $S_t$. As noted in the proof of Theorem \ref{th-L1}, we get the estimate \eqref{mk-sin} instead of \eqref{mk}. This changes the estimate of $I_1$ as
$$
|I_1(-R,0)| \le C_{n,\psi}\,\frac1t
$$
for $t$ large enough. The other estimates are as in Case 1, so that we arrive at
$$
|k_t(-R,0)| \ge C_{n,\psi} \Big(1 - e^{-2t} - \gamma_l\, e^{-lt} - \frac1t \Big),
$$
which is bounded from below by a constant $C_{n,\psi}>0$ as $t\to+\infty$.
\end{proof}

As pointed out in \cite{mueller-thiele}, this shows that no dispersive-type $L^1-L^\infty$ estimates hold for the wave equation.

\subsection*{Acknowledgments}

Yu. K. thanks Professor Waldemar Hebisch for valuable discussions on the general context of multipliers on Lie groups.

 This work was started during an ICL-CNRS fellowship of the second named author at the Imperial College London. Yu.~K. is supported by the ANR-19-CE40-0002 grant of the French National Research Agency (ANR). H.~Z. is supported by the European Union's Horizon 2020 research and innovation programme under the Marie Sk\poll odowska-Curie grant agreement No.~754411 and the Lise Meitner fellowship, Austrian Science Fund (FWF) M3337. 
R. A. was supported by the EPSRC grant EP/R003025.
M. R. is supported by the EPSRC grant EP/R003025/2 and by the FWO Odysseus 1 grant G.0H94.18N: Analysis and Partial Differential Equations.


\end{document}